\documentclass[11pt]{article}

\usepackage[margin=1in]{geometry}

\usepackage[utf8]{inputenc}
\usepackage{comment}
\usepackage{relsize}
\usepackage{hyperref}
\usepackage{tikz}
\usepackage{amssymb}
\usepackage{amsthm} 
\usepackage{mathtools}
\usepackage{lmodern}
\usepackage{verbatim, graphicx}
\usepackage{stmaryrd}
\usepackage{shuffle}
\hypersetup{
	colorlinks=true,
	linkcolor=blue,
	citecolor=blue}

\newtheorem{theorem}{Theorem}[section]
\newtheorem{corollary}[theorem]{Corollary}
\newtheorem{lemma}[theorem]{Lemma}
\newtheorem{proposition}[theorem]{Proposition}
\numberwithin{equation}{section}

\theoremstyle{definition}

\newtheorem*{remark}{Remark}
\newtheorem{example}[theorem]{Example}

\newcommand{\cK}{\mathcal{K}}
\newcommand{\cP}{\mathcal{P}}

\newcommand{\cR}{\mathcal{R}}

\newcommand{\cS}{\mathcal{S}}
\newcommand{\cQ}{\mathcal{Q}}

\newcommand{\cO}{\mathcal{O}}

\newcommand{\Irr}{\mathrm{Irr}}

\newcommand{\Ind}{\mathrm{Ind}}

\newcommand{\dd}{\displaystyle}

\newcommand{\UT}{\mathrm{UT}}

\newcommand{\cN}{\mathcal{N}}

\newcommand{\FF}{\mathbb{F}}

\newcommand{\nn}{\mathrm{nn}}

\newcommand{\f}{\mathrm{f}}
\newcommand{\Cl}{\mathtt{Cl}}
\newcommand{\Ch}{\mathtt{Ch}}
\newcommand{\tS}{\mathsf{S}}
\newcommand{\CC}{\mathbb{C}}

\newcommand{\GL}{\mathrm{GL}}
\newcommand{\opart}{\mathrm{pp}}
\newcommand{\nnopart}{\mathrm{pp}_\nn}
\newcommand{\cL}{\mathcal{L}}

\newcommand{\Res}{\mathrm{Res}}

\newcommand{\spanning}{\text{-span}}
\newcommand{\Inf}{\mathrm{Inf}}

\newcommand{\Int}{\mathrm{Int}}

\newcommand{\One}{1\hspace{-.13cm}1}
\newcommand{\cA}{\mathcal{A}}

\newcommand{\cC}{\mathcal{C}}
\newcommand{\cM}{\mathcal{M}}

\newcommand{\pp}{\mathrm{lbl}}

\newcommand{\poset}{\mathrm{PO}}
\newcommand{\pattern}{\mathbf{ptt}}

\newcommand{\cJ}{\mathcal{J}}
\newcommand{\anti}{\mathrm{Anti}}
\newcommand{\poposet}{\mathbf{ppst}}
\newcommand{\fpoposet}{\mathbf{Tppst}}

\newcommand{\Fac}{\mathrm{Fac}}
\newcommand{\select}{\mathrm{slt}}
\newcommand{\SC}{\mathrm{SC}}
\newcommand{\NPL}{\mathrm{NPtt}}
\newcommand{\ZZ}{\mathbb{Z}}
\newcommand{\width}{\mathrm{wdth}}

\title{Pattern groups and a\\ poset based Hopf monoid }

\date{}
\author{Farid Aliniaeifard and Nathaniel Thiem}

\begin{document}
	
	\maketitle
	
\begin{abstract}
The supercharacter theory of algebra groups gave us a representation theoretic realization of the Hopf algebra of symmetric functions in noncommuting variables.  The underlying representation theoretic framework comes equipped with two canonical bases, one of which was completely new in terms of symmetric functions.  This paper simultaneously generalizes this Hopf structure by considering a larger class of groups while also restricting the representation theory to a more combinatorially tractable one.  Using the normal lattice supercharacter theory of pattern groups, we not only gain a third canonical basis, but also are able to compute numerous structure constants in the corresponding Hopf monoid, including coproducts and antipodes for the new bases.
\end{abstract}

\section{Introduction} \label{Introduction}

A supercharacter theory of the unipotent upper-triangular matrices of the finite general linear groups gave a representation theoretic interpretation to the Hopf algebra of symmetric functions in non-commuting variables \cite{AABB12}.  In fact, these supercharacter theories glue together most naturally as a Hopf monoid as described in \cite{ABT13}, where we obtain the Hopf algebra as a quotient structure.   They give a rich combinatorics on set partitions explored  in \cite{BBT13,BT15,BT16}.   However, the overall Hopf structure remains mysterious, especially with regard to the coproduct on the character basis.  This paper explores a supercharacter theory that is simultaneously computable for a larger swath of groups (including in our case all pattern groups), and yet is more amenable to explicit computation of structure constants.

Formally introduced by Diaconis--Isaacs \cite{DI08}, a supercharacter theory can be thought of as an approximation to the usual character theory of a group.  Given any set partition $\cK$ of a group $G$, one can study the subspace of the space of functions $\f(G)=\{\psi:G\rightarrow\CC\}$ that are constant on the parts of $\cK$.  If this subspace additionally has a basis of orthogonal characters, then we say that the parts of $\cK$ are the superclasses of a supercharacter theory.  The interplay between the parts of $\cK$ and the basis of characters (called supercharacters) mimic the interplay between conjugacy classes and irreducible characters.    This point of view then gives a framework for studying the representation theory of coarser (and often more combinatorial) partitions of groups.

This paper uses a specific supercharacter theory introduced by Aliniaeifard \cite{AL17} in his Ph.D. thesis work.  It gives a general construction for arbitrary groups (though it prefers groups with non-trivial normal subgroups) that has many combinatorial properties baked in.   The paper \cite{AT18} explores some more combinatorial implications of such theories in general, giving lattice-based formulas for the supercharacter values and for the restriction of supercharacters.  This paper applies these techniques to the case of pattern groups (which is in fact the original motivation for the work).

Pattern groups are a family of unipotent groups that are built out of finite posets, roughly a group version of an incidence algebra.  While they were a fundamental example in \cite{DI08}, the supercharacter theory they give in that paper for pattern groups is not generally well understood.  Andrews introduced a different supercharacter theory called a non-nesting supercharacter theory that has nice combinatorial properties \cite{An15}; in fact, \cite{AT17} used this theory to study generalize Gelfand--Graev characters for the finite general linear groups.  While his theory differs from ours, it is in fact morally equivalent.   This paper explores a Hopf monoid first defined (up to moral equivalence) by Andrews; however, the Aliniaeifard supercharacter theory gives us some additional tools, including a third canonical basis and a restriction formula from \cite{AT18}.  These tools allow us to give more explicit structural results, including a coproduct on supercharacters and antipode formulas.
  
After reviewing some of the background material in Section \ref{Preliminaries}, we apply the results of \cite{AT18} to the pattern group case in Section \ref{PatternGroupCombinatorics}.  The main result of this section, Corollary \ref{PatternRestriction}, gives a combinatorial formula for restricting between pattern groups.  Section \ref{Monoid} reviews the monoid constructed in \cite{An15} and examines the structure constants of various bases.  Here we have a new third basis coming from the normal lattice supercharacter theory and Theorem \ref{StructureConstantsSupercharacters} gives a formula for the coproduct on supercharacters.  We then explore some of the structure of this monoid, establish an algebraically independent set of free generators (as a monoid), and construct the primitive elements in the style of \cite{BBT13}.  Along the way, we also compute the antipode on several bases.

\section{Preliminaries}\label{Preliminaries}

This section reviews the necessary background on pattern groups, supercharacter theories, and distributive lattices.  Throughout we will make use of different posets on an underlying set, so given a set $A$, let
$$\poset(A)=\{\text{partial orders on $A$}\}.$$

\subsection{Pattern groups}

Fix a finite field $\FF_q$, a set $A$, and a poset $\cR\in \poset(A)$, let 
$$\Int(\cR)=\{[i,j]\mid i,j\in \cR, i\preceq_\cR j\}$$
be the interval poset of $\cR$ (ordered by inclusion).  The corresponding \textbf{\emph{pattern group}} is given by
$$\UT_\cR=\{u:\Int(\cR)\rightarrow \FF_q\mid u([i,i])=1, i\in A\},$$
where for $u,v\in \UT_\cR$,
$$(uv)([i,k])=\sum_{i\preceq_\cR j\preceq_\cR k} u([i,j])v([j,k]).$$
In the case where $\cR$ is a linear order, we obtain the maximal group of upper triangular matrices in the finite general linear group with rows and columns indexed by $\cR$.

For each subposet $\cO\in\poset(A)$ of $\cR$, we have that $\UT_\cO\subseteq \UT_\cR$.  Let
$$\Int^\circ(\cR)=\{[i,j]\mid i,j\in \cR, i\prec_\cR j\}$$
be the set of proper intervals.  Since this paper is concerned with normal subgroups, we would like an easy characterization of when these subgroups $\UT_\cO$ is in fact normal. Recall, that a \textbf{\emph{co-ideal}} $I$ in a poset $\cR$ is a subset that satisfies $K\in I$ implies $L\in I$ for all $L\succeq_\cR K$.

\begin{proposition}
For subposet $\cO$ of $\cR$, $\UT_\cO\triangleleft\UT_\cR$ if and only if $\Int^\circ(\cO)$ is a co-ideal of $\Int^\circ(\cR)$.
\end{proposition}
\begin{proof}
Suppose $\UT_\cO\triangleleft\UT_\cR$. Let $[j,k]\in \Int^\circ(\cO)$ and suppose $[i,l]\in \Int^\circ(\cR)$ with $i\preceq_\cR j\prec_\cR k\preceq_{\cR}l$.  For $a,b\in \cR$ with $a\prec_\cR b$ and $t\in \FF_q$, define $e_{ab}(t)\in \UT_\cR$ by
\begin{equation}\label{PatternGenerators}
e_{ab}(t)([c,d])=\left\{\begin{array}{ll} 1 & \text{if $c=d$,}\\ t & \text{if $c=a,d=b$,}\\ 0 &\text{otherwise.}\end{array}\right.
\end{equation}
Then
$$e_{ik}(rs)=e_{ij}(r)e_{jk}(s)e_{ij}(-r)e_{jk}(-s)\in \UT_\cO, e_{il}(rst)= e_{ik}(rs) e_{kl}(t)e_{ik}(-rs) e_{kl}(-t)\in \UT_\cO,$$
so $i\prec_\cO l$ and $[i,l]\in \Int^\circ(\cO)$.  The converse follows easily from the definition of multiplication in a pattern group.
\end{proof}

\begin{remark}
For each linearization $\cL$ of $\cR$, there is a corresponding inclusion of $\UT_\cR$ in the unipotent group $\UT_\cL$.  However, some of these choices may lead to normal subgroups while others may not.  For example,
$$\UT_{\heartsuit<\clubsuit<\diamondsuit<\spadesuit}\supseteq \left\{\left(\begin{array}{cccc} 
1 & \ast & \ast  & \ast \\
0 & 1 & 0 & 0 \\
0 & 0 & 1& \ast \\
0 & 0 & 0 & 1\end{array}\right)\right\}\cong \UT_{\begin{tikzpicture} [scale=.3]
\foreach \x/\y/\z/\w in {0/0/\heartsuit/1,-.5/1/\diamondsuit/2,.5/1/\clubsuit/3,-.5/2/\spadesuit/4}
	\node (\w) at (\x,\y)[inner sep=0pt] {$\scriptscriptstyle\z$};
\foreach \x/\y in {1/2,1/3,2/4}
	\draw (\x) -- (\y);
\end{tikzpicture}}\cong 
\left\{\left(\begin{array}{cccc} 
1 & \ast & \ast &\ast  \\
0 & 1 & 0 & \ast \\
0 & 0 & 1& 0\\
0 & 0 & 0 & 1\end{array}\right)\right\}\subseteq \UT_{\heartsuit<\diamondsuit<\clubsuit<\spadesuit}.$$
The left-hand inclusion is as a non-normal subgroup and the right-hand one is as a normal subgroup.  
\end{remark}

 Given a poset $\cR$, the pattern group $\UT_\cR$ generally has many normal subgroups, and in general it is not clear that it is possible to find them all.  However, we are primarily interested in the subset of normal subgroups
\begin{equation}\label{NormalPatternSubgroups}
\NPL(\cR)=\{\UT_\cQ\triangleleft \UT_\cR\mid \cQ\in \poset(A)\text{ a subposet of }\cR\}.
\end{equation}
These subgroups are well-behaved in the sense that 
$$\UT_\cO\cap \UT_\cP=\UT_{\cO\cap \cP} \quad \text{where}\quad i\preceq_{\cO\cap \cP}j \text{ if and only if  $i\preceq_\cO j$ and $i\preceq_\cP j$},$$
and 
$$ \UT_\cO \UT_\cP =\UT_{\cO\cup\cP} \quad \text{where}\quad i\preceq_{\cO\cup \cP}j \text{ if and only if  $i\preceq_\cO j$ or $i\preceq_\cP j$}.$$
Thus, we obtain a poset with unique greatest lower bounds and smallest upper bounds.  In Section \ref{LatticeIntro}, we will show that $\NPL(\cR)$ in fact forms a distributive lattice.

\subsection{Normal lattice supercharacter theories}\label{NormalLatticeSupercharacter}

Given a set partition $\cK$ of $G$, let
$$\f(G;\cK)=\{\psi:G\rightarrow \CC\mid \{g,h\}\subseteq K\in \cK\text{ implies } \psi(g)=\psi(h)\}$$
be the set of functions constant on the blocks of $\cK$.

A \textbf{\emph{supercharacter theory}} $\tS$ of a finite group $G$ is a pair $(\Cl(\tS),\Ch(\tS))$ where $\Cl(\tS)$ is a set partition of $G$ and $\Ch(\tS)$ is a set partition of the irreducible characters $\Irr(G)$ of $G$, such that
\begin{enumerate}
\item[(SC1)] $\{1\}\in \Cl(\tS)$,
\item[(SC2)] $|\Cl(\tS)|=|\Ch(\tS)|$,
\item[(SC3)] For each $X\in \Ch(\tS)$,
$$\sum_{\psi\in X}\psi(1)\psi\in \f(G;\Cl(\tS)).$$
\end{enumerate}  

We refer to the blocks of $\Cl(\tS)$ as the \textbf{\emph{superclasses}} of $\tS$, and the elements of 
$$\{\sum_{\psi \in X}\psi(1)\psi\mid X\in \Ch(\tS) \}$$
 as \textbf{\emph{supercharacters}} of $\tS$.   In fact, the supercharacters of $\tS$ will form an orthogonal basis for $\f(G;\Cl(\tS))$; thus, the superclasses are unions of conjugacy classes.  In general, one only needs to specify one partition ($\Cl$ or $\Ch$) since the two determine one-another.
 
\begin{lemma}[\cite{DI08}]
If $\Cl$ is a set partition of $G$, then there is at most one set partition $\Ch$ of $\Irr(G)$ such that $(\Cl,\Ch)$ is a supercharacter theory.
\end{lemma}

The trivial examples of supercharacter theories have partitions 
\begin{equation*} \Cl=\{\text{conjugacy classes}\}\quad \text{and}\quad
\Cl=\{\{1\},G-\{1\}\}.
\end{equation*}
 
There are several standard supercharacter theories that get applied to pattern groups.  Since a pattern group $\UT_\cR$ is in fact an algebra group, \cite{DI08} defines a supercharacter theory whose superclasses are the equivalence classes of the set partition
\begin{equation}\label{AlgebraGroup}
u\sim v\quad \text{if and only if} \quad \text{there exist $a,b\in \UT_\cR$ such that $u=1+a(v-1)b$.}
\end{equation}
However, it is not even known whether this theory is in general wild, and they certainly are not generally understood.  Andrews \cite{An15} defines a more suitable theory which he calls a non-nesting supercharacter theory for $\UT_\cR$.   Given $u\in \UT_\cR$, let $N_u$ be the smallest normal pattern subgroup containing $u$.  Then the superclass containing $u$ is
$$\Cl(u)=u\bigcap_{M\text{ maximal}\atop \text{in } N_u} M.$$
This theory is very close to the supercharacter theory that this paper focuses on. 

The following supercharacter theory defined in \cite{AL17} begins with a set of normal subgroups and constructs a supercharacter theory that has these normal subgroups as unions of superclasses.

\begin{theorem}[Normal Lattice Supercharacter Theory \cite{AL17}]\label{NormalLatticeSupercharacterTheory}
 Let $\cN$ be a set of normal subgroups such that
 \begin{enumerate}
 \item[(1)] $\{1\},G\in \cN$,
 \item[(2)] $M,N\in \cN$ implies $MN,M\cap N\in \cN$.
 \end{enumerate}   
 Then the partitions
$$\Cl=\{N_\circ\neq \emptyset\mid N\in \cN\},\quad \text{where} \quad N_\circ=\{g\in N\mid g\notin M\in \cN, \text{ if $N\in \cC(M)$}\},$$
and 
$$\Ch=\{X^{N^\bullet}\neq \emptyset \mid N\in \cN\},\quad \text{where} \quad  X^{N^\bullet}=\{\psi\in \Irr(G)\mid N\subseteq \ker(\psi)\nsupseteq O,\text{ if $O\in \cC(N)$}\}$$
define a supercharacter theory $\tS_\cN$ of $G$.
\end{theorem} 

In the case of pattern subgroups, we let 
$$\cN=\NPL(\cR)$$
or the set of all normal pattern subgroups of $\UT_\cR$.  In this case, from the definition of the superclasses it is not difficult to see that $\tS_{\NPL(\cR)}$ is a coarsening of Andrews' non-nesting theory. Write
\begin{equation}\label{PatternNormalLattice}
\tS_{\NPL(\cR)}=(\Cl(\cR),\Ch(\cR)).
\end{equation}

\subsection{Distributive lattices}\label{LatticeIntro}

In order to study normal lattice supercharacter theories we select lattices of normal pattern subgroups, and in our case they are in fact distributive lattices.  

In a poset $\cL$, a \textbf{\emph{cover}} $L\in \cL$ of an element $K\in \cL$ is an element that satisfies $K\prec_{\cL} L$ and if $K\prec_{\cL}M\preceq_{\cL} L$ then $L=M$.  Let
$$\cC(K)=\{\text{covers of $K$ in $\cL$}\}.$$
A \textbf{\emph{distributive lattice}} $\cL$ is a poset such that if 
$$\cM=\{L\in \cL\mid |\cC(L)|=1\}$$
is the set of \textbf{\emph{meet irreducible}} elements, then for each $K\in \cL$ there exists a unique anti-chain $\cA\subseteq \cM$ such that $K$ is the unique largest element in $\cL$ smaller than all elements in $\cA$.   In this case, we write
$$\cL=J^\vee(\cM).$$

In fact, this is equivalent to if
$$\cJ=\{L\in \cL\mid \#\{K\in \cL\mid L\in \cC(K)\}=1\},$$
is the set of \textbf{\emph{join irreducible}} elements, then  for each $K\in \cL$ there exists a unique anti-chain $\cA\subseteq \cJ$ such that $K$ is the unique smallest element in $\cL$ bigger than all elements in $\cA$.  In this case we write,
$$\cL=J(\cJ).$$

\begin{lemma}
For each poset $\cR$, the poset $\NPL(\cR)$ is a distributive lattice.
\end{lemma}
\begin{proof}
Let $\cR$ be a poset.  Each subposet $\cQ$ such that $\UT_\cQ\triangleleft \UT_\cR$ corresponds to a co-ideal $\Int(\cQ)$ in $\Int(\cR)$. A cover of $\UT_\cQ$ then corresponds to adding a maximal element in $\Int^\circ(\cR)-\Int^\circ(\cQ)$ to $\Int^\circ(\cQ)$.  Thus the the set of meet irreducible elements is
$$\cM=\{\UT_{\cR_{[i,j]}}\mid [i,j]\in \Int^\circ(\cR)\},$$
where 
$$ k\prec_{\cR_{[i,j]}} l\quad \text{if and only if} \quad [k,l]\npreceq_{\Int^\circ(\cR)} [i,j].$$
That is, each element in $\cM$ corresponds to a principal ideal in $\Int^\circ(\cR)$.  Then if $\UT_\cP\triangleleft\UT_\cR$, then $\Int^\circ(\cP)$ is a co-ideal in $\Int^\circ(\cR)$.  Let $\lambda$ be the set of maximal elements of $\Int^\circ(\cR)$ not in $\Int^\circ(\cP)$.  Then
$$\UT_\cP=\bigcap_{[i,j]\in \lambda} \UT_{\cR_{[i,j]}}.\qedhere$$
\end{proof}

For $\cA\subseteq \NPL(\cR)$, let
$$\underline{\cA}=\bigcap_{U\in \cA} U\quad \text{and}\quad \overline{\cA}=\prod_{U\in \cA} U.$$
By the previous lemma, for each $\UT_\cQ\triangleleft \UT_\cR$, there is a unique anti-chain $\cA\subseteq \cM$ such that
$$\UT_\cQ=\underline{\cA}.$$

\begin{remark}
The join irreducibles are also easily computed.   By removing any minimal element in $\Int^\circ(\cQ)$, we get a normal subgroup that has $\UT_\cQ$ as a cover.  Thus,
$$\cJ=\{\UT_{\cR_{[i,j]}^\vee}\mid [i,j]\in \Int^\circ(\cR)\},$$
where
$$ k\prec_{\cR_{[i,j]}^\vee} l\quad \text{if and only if} \quad [i,j] \preceq_{\Int^\circ(\cR)}[k,l].$$
That is, each element in $\cM$ corresponds to a principal co-ideal in $\Int^\circ(\cR)$.
\end{remark}

\section{The poset partition combinatorics of a normal lattice theory}\label{PatternGroupCombinatorics}

This section applies \cite{AT18} to the particular case of the normal lattice theory built on normal pattern subgroups.  We first review the poset combinatorics introduced in \cite{An15} and connect it to the normal pattern subgroup lattice $\NPL(\cR)$.    We review the character formula of this theory from the point of view introduced in \cite{AT18}, and conclude with a combinatorial description of restriction between pattern groups.

\subsection{Poset partitions}

Fix a set $A$, and let $\cR\in \poset(A)$.  An $\cR$-\textbf{\emph{partition}} is a subset
$$\lambda\subseteq \Int^\circ(\cR)$$
such that the interval between two elements $[i,j],[k,l]\in \lambda$ is a total order if and only if $[i,j]=[k,l]$.
Let
\begin{equation}
\opart(\cR)=\{\cR\text{-partitions}\}.
\end{equation}

We may visualize these partitions by placing the Hasse diagram of $\cR$ as the base, and then place an arc from $i$ to $j$ if $[i,j]\in \lambda$.  For example, 
$$\lambda=\{[1,4],[4,5],[3,5]\}\in \opart\left(
\begin{tikzpicture}[scale=.3,baseline=.5cm]
\foreach \x/\y/\z in {0/0/1,0/1/2,-1/2/3,1/2/4,0/3/5,0/4/6}
	\node (\z) at (\x,\y) [inner sep = 0] {$\scriptscriptstyle{\z}$};
\foreach \x/\y in {1/2,2/3,2/4,3/5,4/5,5/6}
	\draw[gray] (\x)--(\y);
\end{tikzpicture}\right)\quad \text{is}\quad 
\begin{tikzpicture}[scale=.7,baseline=0cm]
\foreach \x/\y/\z in {0/0/1,1/.1/2,1.9/.6/3,2.1/-.4/4,3/.3/5,4/.4/6}
	\node (\z) at (\x,\y) [inner sep = 0] {$\scriptscriptstyle{\z}$};
\foreach \x/\y in {1/2,2/3,2/4,3/5,4/5,5/6}
	\draw[gray] (\x)--(\y);
\foreach \x/\y in {1/4,4/5,3/5}
	\draw (\x) to [out=70, in=110] (\y);
\end{tikzpicture}\ .
$$
Note that if $\cL$ is a linear order, then the transitive closure of the relation $i\sim_\lambda j$ if $[i,j]\in \lambda$ gives a set partition of the underlying set.

Given a poset $\cP$, let
$$\anti(\cP)=\{\lambda\subseteq \cP\mid \lambda\text{ is an antichain}\}.$$
An $\cR$-partition $\lambda$ is \textbf{\emph{non-nesting}} if $\lambda$ is an anti-chain in $\Int^\circ(\cR)$.   Let
$$\nnopart(\cR)=\{\lambda\in\opart(\cR)\mid \lambda\in \anti\Big(\Int^\circ(\cR)\Big)\}.$$
For example,
$$\begin{tikzpicture}[scale=.7,baseline=0cm]
\foreach \x/\y/\z in {0/0/1,1/.1/2,1.9/.6/3,2.1/-.4/4,3/.3/5,4/.4/6}
	\node (\z) at (\x,\y) [inner sep = 0] {$\scriptscriptstyle{\z}$};
\foreach \x/\y in {1/2,2/3,2/4,3/5,4/5,5/6}
	\draw[gray] (\x)--(\y);
\foreach \x/\y in {3/6,4/5}
	\draw (\x) to [out=70, in=110] (\y);
\end{tikzpicture}\in \nnopart\left(
\begin{tikzpicture}[scale=.3,baseline=.5cm]
\foreach \x/\y/\z in {0/0/1,0/1/2,-1/2/3,1/2/4,0/3/5,0/4/6}
	\node (\z) at (\x,\y) [inner sep = 0] {$\scriptscriptstyle{\z}$};
\foreach \x/\y in {1/2,2/3,2/4,3/5,4/5,5/6}
	\draw[gray] (\x)--(\y);
\end{tikzpicture}\right)\qquad \text{but}\qquad \begin{tikzpicture}[scale=.7,baseline=0cm]
\foreach \x/\y/\z in {0/0/1,1/.1/2,2/.2/3,3/.3/4,4/.4/5,5/.5/6}
	\node (\z) at (\x,\y) [inner sep = 0] {$\scriptscriptstyle{\z}$};
\foreach \x/\y in {1/2,2/3,2/4,3/5,4/5,5/6}
	\draw[gray] (\x)--(\y);
\foreach \x/\y in {3/6,4/5}
	\draw (\x) to [out=70, in=110] (\y);
\end{tikzpicture}\notin \nnopart\left(
\begin{tikzpicture}[scale=.3,baseline=.6cm]
\foreach \x/\y/\z in {0/0/1,0/1/2,0/2/3,0/3/4,0/4/5,0/5/6}
	\node (\z) at (\x,\y) [inner sep = 0] {$\scriptscriptstyle{\z}$};
\foreach \x/\y in {1/2,2/3,3/4,4/5,5/6}
	\draw[gray] (\x)--(\y);
\end{tikzpicture}\right).$$

\begin{proposition} \label{NonnestingToPattern}
Let $\cR$ be poset.  Then
\begin{enumerate}
\item[(a)] The function 
$$\begin{array}{r@{\ }c@{\ }c@{\ }c} \pp_\Cl: & \{\UT_\cQ\triangleleft \UT_\cR\} & \longrightarrow & \nnopart(\cR)\\
& \UT_\cQ & \mapsto & \left\{\begin{array}{@{}c@{}}\text{minimal elements}\\ \text{of $\Int^\circ(\cQ)$}\end{array}\right\}.
\end{array}
$$
is a bijection.
\item[(b)] The function 
$$\begin{array}{r@{\ }c@{\ }c@{\ }c} \pp_\Ch:  & \nnopart(\cR)  & \longrightarrow & \{\UT_\cQ\triangleleft \UT_\cR\}\\
&  Q & \mapsto &\dd\bigcap_{[i,j]\in Q}  \UT_{\cR_{[i,j]}} 
\end{array}
$$
is a bijection.
\item[(c)] For $\UT_\cR$ with normal pattern group sub-lattice, the superclasses and supercharacters are indexed by $\nnopart(\cR)$.
\end{enumerate}
\end{proposition}

\begin{remark} The functions in (a) and (b) do not invert one-another.  However, they are both representation theoretically useful.  Thus for $\lambda\in \nnopart(\cR)$, let
\begin{equation}\label{PartitionToSubgroup}
\begin{split}
\UT^\lambda &= \pp_\Cl^{-1}(\lambda)=\prod_{[i,j]\in \lambda} \UT_{\cR^{\vee}_{[i,j]}}\\
\UT_\lambda&=  \pp_\Ch(\lambda)=\bigcap_{[i,j]\in \lambda}  \UT_{\cR_{[i,j]}} .
\end{split}
\end{equation}
In other words,  in the first case, $\lambda$ specifies the minimal elements in co-ideal of $\Int^\circ(\cR)$, and in the second $\lambda$ specifies the maximal elements not in the co-ideal of $\Int^\circ(\cR)$.
\end{remark}

\begin{example}
When $\cR\in \poset(B)$ is a total order, then $|\nnopart(\cR)|$ is the $|B|$th Catalan number.  In fact, if we view $\UT_\cR$ as upper-triangular matrices, then there is a natural Dyck path for each normal subgroup.  If we let $\ast$ indicate entries that may be nonzero, one such subgroup might be
$$\UT_\lambda=\left[
\begin{tikzpicture}[scale=.5,baseline=2.15cm]
\foreach \x in {1,...,8}
	\node (\x\x) at (\x,9-\x) {$1$};
\foreach \x in {1,...,7}
	{\foreach \y in {\x,...,7}
		\node (\x\y) at (\x,8-\y) {$0$};}
\draw[gray,thick] (.5,8.5) -- (1.5,8.5) -- (1.5,7.5) -- (4.5,7.5) -- (4.5,5.5) -- (5.5,5.5) -- (7.5,5.5) -- (7.5,2.5) -- (8.5,2.5) -- (8.5,.5);
\foreach \x/\y in {3/7,4/7,4/6,5/5,6/5,7/5,6/4,7/4,7/3,8/2}
	\node at (\x,\y) {$0$};
\foreach \x in {2,...,8}
	\node at (\x,8) {$\ast$};
\foreach \y in {6,7}
	{\foreach \x in {5,...,8}
		\node at (\x,\y) {$\ast$};}
\foreach \y in {3,4,5}
	\node at (8,\y) {$\ast$};
\end{tikzpicture}\right]=\UT^\mu.$$
The set partitions $\lambda$ and $\mu$ capture aspects of the Dyck path.  That is, $\lambda$ gives the coordinates of the peaks, so in the example
$\lambda=\{[2,4],[4,7],[7,8]\}$ (where we omit peaks with trivial coordinates), and $\mu$ gives the coordinates of the valleys, so $\mu=\{[1,2],[3,5],[6,8]\}$.
\end{example}

\subsection{The supercharacters of $\UT_\cR$}

By Theorem \ref{NormalLatticeSupercharacterTheory}, $\NPL(\cR)$ corresponds to a normal lattice supercharacter theory $\tS_\cR=\tS_{\NPL(\cR)}$.    By Proposition \ref{NonnestingToPattern}, we have that $\Cl$ and $\Ch$ are indexed by $\nnopart(\cR)$. However, we do this indexing slightly asymmetrically using Proposition \ref{NonnestingToPattern} (a) and (b) for combinatorial reasons.  Let $\lambda,\mu\in \nnopart(\cR)$.   For superclasses, let
$$ \UT^\mu_\circ\quad \text{be the superclass corresponding  to $\mu$,}$$
and for supercharacters, let
\begin{equation}\label{PatternSupercharacterDefinition}
\chi^\lambda=\chi^{\UT_\lambda^\bullet}.
\end{equation}
These different choices make the computation of relevant statistics more straight-forward; for $\UT_\lambda$, the poset partition $\lambda$ identifies the covers of $\UT_\lambda$ and for $\UT^\mu_\circ$, the poset partition $\mu$ gives the entries required to be nonzero to be found in that superclass (or the subgroups that $\UT^\mu$ covers).  In more detail,
$$\cC(\UT_\lambda)=\Big\{\UT_\lambda^{+[i,j]} \mid [i,j]\in \lambda\Big\}\quad \text{where}\quad \UT_\lambda^{+[i,j]}=\{e_{[i,j]}(t)\mid t\in \FF_q\}\UT_\lambda,$$
so
$$\overline{\cC(\UT_\lambda)}=\UT_\lambda\cdot\prod_{[i,j]\in \lambda} \{e_{[i,j]}(t)\mid t\in \FF_q\}.$$
Similarly, if $g\in \UT^\mu_\circ$, then $g\in \overline{\cC(\UT_\lambda)}$ if and only if $[j,k]\in \mu, [i,l]\in \lambda$ with $i\preceq j\prec k\preceq l$ implies $i=j$ and $k=l$.  In fact, if  $g\in \overline{\cC(\UT_\lambda)}$, then
$$g\in \overline{\{ \UT_\lambda^{+[i,j]}\mid [i,j]\in \lambda\cap \mu\}}.$$

Applying \cite[Corollary 3.4]{AT18}, we obtain the following character formula.
\begin{proposition}\label{PatternCharacterFormula}Let $\lambda,\mu\in \nnopart(\cR)$.  For $g\in \UT^\mu_\circ$,
$$\chi^{\lambda}(g)=\left\{\begin{array}{ll}\frac{|\UT_\cR|}{|\UT_\lambda|}\Big(1-\frac{1}{q}\Big)^{|\lambda|} \Big(\frac{1}{1-q}\Big)^{|\lambda\cap \mu|} & \text{if $g\in \overline{\cC(\UT_\lambda)}$,}\\ 0 & \text{otherwise.}
\end{array}\right.$$
\end{proposition}

\begin{remark}
This is very close to a character formula given by \cite{An15} for his non-nesting supercharacter theories for pattern groups. In fact, one may obtain the normal lattice theory supercharacter theory from Andrews' theory be conjugating the superclasses by diagonal matrices.
\end{remark}

\subsection{A restriction formula}

One motivation of this paper is to sort out the implications of the restriction rule in \cite{AT18} for pattern groups.    Let $H\subseteq G$ be a subgroup and suppose we have sublattices of normal subgroups $J^\vee(\cM_G)$ and $J^\vee(\cM_H)$ that are distributive lattices.  We say  $J^\vee(\cM_G)$ and $J^\vee(\cM_H)$ are \textbf{\emph{restriction favorable}} if
\begin{enumerate}
\item[(R1)]  The function
$$\begin{array}{r@{\ }c@{\ }c@{\ }c}\cdot\cap H:&  J^\vee(\cM_G) & \longrightarrow & J^\vee(\cM_H)\\
&N & \mapsto & N\cap H
\end{array}$$
is well-defined (that is, $N\cap H\in J^\vee(\cM_H)$),
\item[(R2)] If $M,N\in J^\vee(\cM_G)$ with $N\in \cC(M)$, then either $M\cap H=N\cap H$ or $N\cap H\in \cC(M\cap H)$.
\end{enumerate}

\begin{theorem}{\cite[Corollary 3.11 and Theorem 3.13]{AT18}}  \label{RestrictionFormula} Let $H\subseteq G$, and suppose $J^\vee(\cM_G)$ and $J^\vee(\cM_H)$ are restriction favorable.  Then for an antichain $\cA\subseteq \cM_G$,
\begin{enumerate}
\item[(a)]  There is a bijection $\cC(\underline{\cA})\rightarrow \cA$ where a cover $O$ corresponds to the unique element $P_O$ in $\cA$ such that $P_O\cap O=\underline{\cA}$.  
 \item[(b)] The restriction
$$\frac{\Res^{G}_{H}(\chi^{\underline{\cA}^{\bullet}})}{\chi^{\underline{\cA}^\bullet}(1)}=\sum_{\underline{\cA_H} \supseteq K\supseteq \overline{\cC(\underline{\cA})}\cap\underline{\cA_H}}\frac{|\overline{\cC(\underline{\cA})}\cap \underline{\cA_H} |(-1)^{|\{Q\in \cC(K)\mid Q\subseteq \underline{\cA_H}\}|}}{|\overline{\cC(K)}\cap \underline{\cA_H}|\chi^{K^\bullet}((\overline{\cC(K)}\cap \underline{\cA_H})_\circ)}\chi^{K^\bullet},$$
where $\cA_H=\{P_{O\cap H}\mid O\in \cC(\underline{\cA}), O\cap H\neq \underline{\cA}\cap H\}$.
\end{enumerate}
 \end{theorem}

In our case, all pattern subgroup containments turn out to be restriction favorable.

\begin{lemma}
Let $\cR,\cQ\in \poset(B)$ with $\UT_\cQ\subseteq \UT_\cR$.  Then $\NPL(\cR)$ and $\NPL(\cQ)$ are restriction favorable.
\end{lemma}
\begin{proof}
Let $\cM_\cR=\cM_{\UT_\cR}$ and $\cM_\cQ=\cM_{\UT_\cQ}$ be the respective meet irreducible elements.  Given $\UT_{\cP}\in \NPL(\cR)$, we have that $\UT_\cP\cap \UT_\cQ\triangleleft \UT_{\cQ}$ by the diamond isomorphism theorem, and since $\UT_\cP\cap \UT_\cQ=\UT_{\cP\cap\cQ}$ it is also a pattern subgroup.  We can conclude that the function $\cdot\cap \UT_\cQ$ is well-defined, giving us (R1).  

Suppose that $\lambda\in \nnopart(\cR)$ and $[i,j]\in \lambda$.  Then
$$\UT_\lambda^{+[i,j]}\cap \UT_\cQ=\UT_\lambda\cap \UT_\cQ\quad \text{if and only if} \quad i\nprec_\cQ j,$$
and if $i\prec_\cQ j$, then $\UT_\lambda^{+[i,j]}\cap \UT_\cQ$ covers $\UT_\lambda\cap \UT_\cQ$, giving (R2).
\end{proof}

 For $\lambda\in \nnopart(\cR)$, let
$$\lambda_\cQ=\{[i,j]\in\lambda\mid i\prec_\cQ j\}\in \nnopart(\cQ).$$
We can now apply Theorem \ref{RestrictionFormula} to our situation.

\begin{corollary} \label{PatternRestriction}
Let $\cQ$ be a subposet of $\cR$.  For $\lambda\in \nnopart(\cR)$,
$$\Res^{\UT_\cR}_{\UT_{\cQ}}(\chi^\lambda)=\frac{1}{|\UT_{\lambda_\cQ}/(\UT_\lambda\cap\UT_\cQ)|}\frac{\chi^\lambda(1)}{\chi^{\lambda_\cQ}(1)}\sum_{\nu\in \nnopart(\cQ)\atop\UT_{\lambda_\cQ}\supseteq \UT_{\nu}\supseteq \UT_{\lambda}\cap \UT_\cQ} \chi^\nu.$$
\end{corollary}
\begin{proof}
We begin by examining the ingredients for Theorem \ref{RestrictionFormula}.  We have,
\begin{align*}
\UT_\lambda&=\bigcap_{[i,j]\in \lambda} \UT_{\cR_{[i,j]}}\\
\cC(\UT_\lambda)&=\{\UT_\cR^{+[i,j]}\mid [i,j]\in \lambda\}\\
\overline{\cC(\UT_\lambda)}&= \UT_\lambda\prod_{[i,j]\in \lambda} \{e_{[i,j]}(t)\mid t\in \FF_q\}\\
\chi^\lambda(1)&=\frac{|\UT_\cR|}{|\UT_\lambda|} \frac{(q-1)^{|\lambda|}}{q^{|\lambda|}}\\
(\UT_\lambda)_{\UT_\cQ}&=\UT_{\lambda_\cQ}\\
\underline{(\UT_\lambda)_{\UT_\cQ}}\cap\overline{\cC(\UT_\lambda)} &=\UT_\lambda\cap\UT_\cQ.
\end{align*}
Thus,
\begin{align*}
\Res^{\UT_\cR}_{\UT_{\cQ}}(\chi^\lambda)&=\chi^\lambda(1)\sum_{(\UT_{\lambda_\cQ}\supseteq \UT_\nu\supseteq \UT_\lambda\cap \UT_\cQ} \frac{|\UT_\lambda\cap \UT_\cQ|}{|\UT_\nu|q^{|\nu-\lambda|}} \frac{(-1)^{|\nu-\lambda|}}{\frac{|\UT_\cQ|}{|\UT_\nu|}q^{|\nu|}(-1)^{|\nu-\lambda|} (q-1)^{|\nu\cap \lambda|}}\chi^\nu\\
&=\chi^\lambda(1)\sum_{\UT_{\lambda_\cQ}\supseteq \UT_\nu\supseteq \UT_\lambda\cap \UT_\cQ} \frac{|\UT_\lambda\cap \UT_\cQ|}{|\UT_\cQ|} \frac{q^{|\lambda\cap\nu|}}{ (q-1)^{|\nu\cap \lambda|}}\chi^\nu.
\end{align*}
Note that $\lambda\cap \nu=\lambda_\cQ$ is independent of $\nu$, so
\begin{equation*}
\Res^{\UT_\cR}_{\UT_{\cQ}}(\chi^\lambda)=\frac{\chi^\lambda(1)}{\chi^{\lambda_\cQ}(1)}\frac{|\UT_\lambda\cap\UT_\cQ|}{|\UT_{\lambda_\cQ}|}\sum_{\UT_{\lambda_\cQ}\supseteq \UT_\nu\supseteq \UT_\lambda\cap \UT_\cQ}\chi^\nu,
\end{equation*}
as desired.
\end{proof}

For $\lambda\in \nnopart(\cR)$ and $\mu\in \nnopart(\cQ)$, consider the subposet of $\Int(\cQ)$ given by
$$\Int^{\lambda}_\mu=\Big\{[i,j]\in \Int^\circ(\cQ)\mid [i,j]\in \bigcap_{[k,l]\in \mu} \cQ_{[k,l]}, [i,j]\notin\bigcap_{[k,l]\in \lambda} \cR_{[k,l]}\Big\}.$$
For example, if $\cR=\{1<2<\cdots<8\}$, $\lambda=\{[2,4],[4,7],[7,8]\}$, $\cQ=\{1<2<\cdots<6\}$, and $\mu=\lambda_\cQ=\{[2,4]\}$ then 
$$\left[
\begin{tikzpicture}[scale=.5,baseline=2.15cm]
\foreach \x in {1,...,8}
	\node (\x\x) at (\x,9-\x) {$1$};
\foreach \x in {1,...,7}
	{\foreach \y in {\x,...,7}
		\node (\x\y) at (\x,8-\y) {$0$};}
\draw[gray,thick] (.5,8.5) -- (1.5,8.5) -- (1.5,7.5) -- (4.5,7.5) -- (4.5,5.5) -- (5.5,5.5) -- (7.5,5.5) -- (7.5,2.5) -- (8.5,2.5) -- (8.5,.5);
\draw[gray, thick, dotted] (.5,8.5) -- (6.5,8.5) -- (6.5,2.5) -- (7.5,2.5) -- (7.5,1.5) -- (8.5,1.5);
\foreach \x/\y in {3/7,4/7,4/6,7/5,7/4,7/3,8/2}
	\node at (\x,\y) {$\cdot$};
\foreach \x/\y in {2/8,3/8,4/8,5/8,6/8,5/7,6/7,5/6,6/6,5/5,6/5,6/4}
	\node at (\x,\y) {$\Box$};
\foreach \x in {2,...,8}
	\node at (\x,8) {$\ast$};
\foreach \y in {6,7}
	{\foreach \x in {5,...,8}
		\node at (\x,\y) {$\ast$};}
\foreach \y in {3,4,5}
	\node at (8,\y) {$\ast$};
\end{tikzpicture}\right]\qquad \text{where} \qquad \begin{array}{l}
\bigcap_{[k,l]\in \lambda} \cR_{[k,l]} \text{ $\ast$ or $\Box\hspace{-.25cm}\ast$\ -entries,}\\\\
\bigcap_{[k,l]\in \mu} \cQ_{[k,l]} \text{ $\Box$ or $\Box\hspace{-.25cm}\ast$\ -entries},\\\\
\Int_\mu^\lambda \text{ $\Box$-entries.}
\end{array}$$

Note that while $\Int_\mu^\lambda$ is a subposet of $\Int(\cQ)$ since it is only closed under meets (and not necessarily joins), it does not necessarily correspond to a subposet of $\cQ$.    However, the poset gives us a combinatorial description of the sum in the restriction.

\begin{lemma}\label{RestrictionCombinatorics} Let $\UT_\cQ\subseteq \UT_\cR$ and $\lambda\in \nnopart(\cR)$.  Then
$$\{\mu\in \nnopart(\cQ)\mid \UT_{\lambda_\cQ}\supseteq \UT_{\mu}\supseteq \UT_{\lambda}\cap \UT_\cQ\}=\Big\{\nu\cup \lambda_\cQ\mid \nu\in \anti\Big(\Int^{\lambda}_{\lambda_\cQ}\Big)\Big\}.$$
\end{lemma}

\begin{proof}
We may re-write
$$\Int_{\lambda_\cQ}^\lambda=\bigg\{[i,j] \in \Int^{\circ}(\cQ)\ \bigg|\   \begin{array}{@{}l@{}} \text{$[i,j]\not\preceq_{\Int^{\circ}(\cQ)} [k,l]$ for all $[k,l]\in \lambda_\cQ$,}\\  \text{$[i,j]\preceq_{\Int^{\circ}(\cR)} [k,l]$ for some $[k,l]\in \lambda$}\end{array}\bigg\}$$ 
and 
$$ \UT_\lambda=\prod_{ [i,j]\not\preceq_{\Int^{\circ}(\cR)} [k,l] \atop \text{~for all~}[k,l]\in \lambda\ } e_{[i,j]}(t) \qquad\text{(with $e_{[i,j]}(t)$ as in (\ref{PatternGenerators}))}.$$ 
For $\nu$ an anti-chain in  $\Int^{\lambda}_{\lambda_\cQ}$, it is clear that $\UT_{\nu\cup \lambda_\cQ} \subseteq \UT_{\lambda_\cQ}$. Also, if $e_{[i,j]}\in \UT_{\lambda}\cap \UT_\cQ$, then $[i,j]\in \Int^\circ (\cQ)$ and $[i,j]\not\preceq_{\Int^{\circ}(\cR)} [k,l] \text{~for all~}[k,l]\in \lambda,$ which implies that $[i,j]\not\preceq_{\Int^{\circ}(\cQ)} [k,l] \text{~for all~}[k,l]\in \lambda_\cQ$. Moreover, if $[i,j]\preceq_{\Int^\circ(\cQ)} [k,l]$ for some $[k,l]\in \nu$, then by the definition of $\nu$, there is $[r,s]\in \lambda$ such that $[k,l]\preceq_{\Int^\circ(\cR)} [r,s].$ Thus, $[i,j]\preceq_{\Int^\circ(\cR)}[r,s]$, a contradiction. Therefore,  $[i,j]\not \preceq_{\Int^\circ(\cQ)} [k,l]$ for all $[k,l]\in \nu$, and so $e_{[i,j]}(t)\in \UT_\nu$. Hence, we showed that $\UT_{\lambda_\cQ}\supseteq \UT_{\nu}\cap \UT_{\lambda_\cQ}\supseteq \UT_{\lambda} \cap \UT_\cQ$.

Conversely, we show that if $\UT_{\lambda_\cQ}\supseteq \UT_{\mu}\supseteq \UT_{\lambda}\cap \UT_\cQ$, then $\mu=\nu\cup \lambda_\cQ$ for some antichain $\nu$ of $\Int_{\lambda_Q}^\lambda$. Suppose that instead $\lambda_\cQ \not\subseteq \mu$ and $[i,j]\in \lambda_\cQ-\mu$. Then for some $[k,l]\in \mu$, $[i,j]\preceq_{\Int^\circ(\cQ)}[k,l]$, and so $e_{[k,l]}(t)\subseteq \UT_\lambda \cap \UT_\cQ$, which means $[k,l]\notin\mu$, a contradiction. Therefore, $\lambda_\cQ\subseteq \mu$. Also, for every $[i,j]\in \mu - \lambda_\cQ$, there is $[k,l]\in \lambda$ such that $[i,j]\preceq_{\Int^\circ(\cR)} [k,l]$, and since $\UT_{\lambda_\cQ}\supseteq \UT_\mu$, $[i,j]\not\preceq_{\Int^\circ(\cQ)}[k,l]$ for all $[k,l]\in \lambda_\cQ$. Therefore $[i,j]\in \Int^\lambda_{\lambda_\cQ}$, and since $\mu \in \nnopart(\cQ)$, $\mu- \lambda_\cQ$ is a antichain in $\Int^\lambda_{\lambda_\cQ}$.
\end{proof}

\section{A Hopf monoid}\label{Monoid}

The goal of this section is to revisit a Hopf monoid defined in \cite{An15} built out of the representation theory of pattern groups.  For a more detailed background on Hopf monoids, we recommend \cite{AM10}.  While there has been more literature on Hopf algebras, it appears that Hopf monoids seem to be especially well-suited to the representation theory of unipotent groups \cite{ABT13}.  As it happens, we can easily recover a corresponding Hopf algebra as a quotient, but the monoid structure allows easier computations than in the Hopf algebra.  

We begin with the definition of the main Hopf monoid $\pattern_\NPL$, and then give the structure constants on the main bases.  We then show that $\pattern_\NPL$ is free as a monoid, and compute the antipode on several of our favorite bases.  We conclude with a construction of the primitive elements. 

\subsection{The pattern group Hopf monoid}
Define a vector species $\pattern:\{\text{sets}\}\rightarrow \{\FF_q\text{-modules}\}$ by
$$\pattern[A]=\bigoplus_{\cR\in \poset(A)} \f(\UT_\cR).$$

Let $\cP\in \poset(A)$ and $\cQ\in \poset(B)$ with $A\cap B=\emptyset$.  The \textbf{\emph{concatenation}}  $\cP.\cQ\in \poset(A\cup B)$ of $\cP$ with $\cQ$ is given by
\begin{equation*}
i\preceq_{\cP.\cQ}j \quad\text{if $i\preceq_\cP j$, $i\preceq_\cQ j$ or $i\in A$ and $j\in B$.}  
\end{equation*}
There is a corresponding projection $\pi_{A,B}: \UT_{\cP.\cQ}\rightarrow \UT_\cP\times \UT_\cQ$ given by
$$\pi_{A,B} (u) ([i,j])=\left\{\begin{array}{ll} 
u([i,j]) & \text{if $i,j\in A$ or $i,j\in B$,}\\
0 & \text{otherwise.}
\end{array}\right.$$ 

Given $\cP\in \poset(C)$ and $A\subseteq C$, the  \textbf{\emph{restriction}} $\cP|_A\in \poset(A)$ of $\cP$ to $A$ is given by
$$i\preceq_{\cP|_A} j\quad \text{if $i\preceq_\cP j$}.$$
There is a corresponding injective function $\iota_{A, B}:\UT_{\cP|_A}\times \UT_{\cP|_B}\rightarrow \UT_\cP$ given by
$$\iota_{A, B}(u,v)([i,j])=\left\{\begin{array}{ll} u([i,j]) & \text{if $i,j\in A$}\\ v([i,j]) & \text{if $i,j\in B$} \\ 0 & \text{otherwise.}\end{array}\right.$$
These constructions give us a product and coproduct on $\pattern$ via
$$\begin{array}{r@{\ }c@{\ }c@{\ }c} 
m_{A,B}:  & \f(\UT_\cP)\otimes  \f(\UT_\cQ) & \longrightarrow & \f(\UT_{\cP.\cQ})\\
& \chi\otimes \psi & \mapsto & (\chi,\psi)\circ \pi_{A,B}\end{array}
\quad\text{and}\quad
\begin{array}{r@{\ }c@{\ }c@{\ }c} 
\Delta_{A,B}:  & \f(\UT_\cP) & \longrightarrow &  \f(\UT_{\cP|_A})\otimes  \f(\UT_{\cP|_B})\\
& \psi & \mapsto & \psi\circ \iota_{A,B},\end{array}$$
where these again give us the functors of inflation and restriction, respectively.  As shown in \cite{An15}, these functions are compatible and give us a Hopf monoid.   Furthermore, every supercharacter theory on pattern groups that are compatible with inflation and restriction give us a sub Hopf monoid.    Our focus for the rest of the paper will be on the sub Hopf monoid
$\pattern_\NPL:\{\text{sets}\}\rightarrow \{\FF_q\text{-spaces}\}$ given by
$$\pattern_\NPL[A]=\bigoplus_{\cR\in \poset(A)} \f(\UT_\cR;\Cl(\cR)),$$
where $\Cl(\cR)$ is the superclass partition given in (\ref{PatternNormalLattice}).

\subsection{Some related monoids}

We already have that $\pattern_\NPL\subseteq \pattern$ as Hopf monoids.  In fact, by varying the underlying supercharacter theory, we obtain various additional monoids (assuming these new supercharacter theories are compatible with restriction and inflation).  In particular, $\pattern_\NPL$ is contained in Andrews' Hopf monoid which is itself a sub Hopf monoid of the monoid corresponding to the usual algebra group supercharacter theory (\ref{AlgebraGroup}).  However, this latter theory remains relatively mysterious for general pattern groups.

Another approach is to find submonoids in $\pattern_\NPL$.  For example, if 
$$\pattern_\poset[A]=\bigoplus_{\cR\in \poset(A)} \CC\spanning\{\chi^{\emptyset_\cR}\},\quad \text{where}\quad \emptyset_\cR=\emptyset\in \nnopart(\cR),$$
then $\pattern_\poset$ is isomorphic to the usual Hopf monoid on posets.  

Given a poset $\cR\in \poset(A)$, recall that the \textbf{\emph{width}} $\width(\cR)$ of $\cR$ is the size of the maximal anti-chain.  For $w\in \ZZ_{\geq 1}$, let 
$$\pattern_\NPL^{(w)}[A]=\bigoplus_{\cR\in \poset(A)\atop \width(\cR)\leq w} \f(\UT_\cR;\Cl(\cR)).$$
Then it follows from the general product and coproduct definitions that $\pattern_\NPL^{(w)}$ is a sub Hopf monoid of $\pattern_\NPL$.  In fact, $\pattern_\NPL^{(1)}$ is a subset of the Hopf monoid $\mathbf{scf}(U)$ in \cite{ABT13} that gives the Hopf algebra $\mathbf{SC}$ isomorphic to the Hopf algebra  of symmetric functions in noncommuting variables $\mathrm{NCSym}$ \cite{AABB12}.  In fact, if $\mathrm{Sym}$ is the Hopf algebra of symmetric functions in commuting variables, then there is a functor $\pi:\mathbf{scf}(U)\rightarrow \mathbf{SC}(U)$ and an isomorphism $\mathrm{ch}:\mathbf{SC}\rightarrow \mathrm{NCSym}$ (sending superclass identifier functions to monomial symmetric functions) and a surjective homomorphism $\mathrm{rlx}:\mathrm{NCSym}\rightarrow \mathrm{Sym}$ (allowing variables to commute).

\begin{proposition}
If $\mathbf{SC}_\NPL$ denotes the subalgebra obtained from $\pattern_\NPL^{(1)}\subseteq \mathbf{scf}(U)$, then
$$\mathrm{Sym}=\mathrm{rlx}\circ \mathrm{ch}(\mathbf{SC}_\NPL).$$
\end{proposition} 

In other words, we have
$$\begin{tikzpicture}[baseline=.5]
\node (A) at (0,1) {$\mathbf{scf}(U)$};
\node[rotate=90] at (0,.5) {$\subseteq$};
\node (B) at (0,0) {$\pattern_\NPL^{(1)}$};
\node (C) at (3,1) {$\mathbf{SC}(U)$};
\node[rotate=90] at (3,.5) {$\subseteq$};
\node (D) at (3,0) {$\mathbf{SC}_\NPL$};
\node (E) at (6,1) {$\mathrm{NCSym}$};
\node[rotate=90] at (6,.5) {$\subseteq$};
\node (F) at (6,0) {$\mathrm{NCSym}_\nn$};
\node (G) at (9,1) {$\mathrm{Sym}$};
\node[rotate=90] at (9,.5) {$=$};
\node (H) at (9,0) {$\mathrm{Sym}$};
\draw[->] (A) -- node[above] {$\pi$}  (C);
\draw[->] (C) -- node[above] {$\mathrm{ch}$}  (E);
\draw[->] (E) -- node[above] {$\mathrm{rlx}$}  (G);
\draw[->] (B) -- node [above] {$\pi$}  (D);
\draw[->] (D) -- node[above] {$\mathrm{ch}$}  (F);
\draw[->] (F) -- node[above] {$\mathrm{rlx}$}  (H);
\end{tikzpicture}\ ,$$
where $\mathrm{NCSym}_\nn=\mathrm{ch}\circ \pi(\pattern_\NPL^{(1)}).$
In \cite{NT}, the authors define the Catalan quasi-symmetric Hopf algebra $\mathsf{CQSym}$ in which each degree has dimension $C_n,$ the $n$th Catalan number. In the last section we show that $\mathrm{NCSym}_\nn$ is isomorphic to $\mathsf{CQSym}$.

\subsection{Standard bases}

As a supercharacter theory, $\tS_\cR$ equips $\f(\UT_\cR;\Cl(\cR))$ with two natural bases
\begin{align*}
\f(\UT_\cR;\Cl(\cR))&=\CC\spanning\{\delta_\mu\mid \mu\in \nnopart(\cR)\}\\
&=\CC\spanning\{\chi^\lambda \mid \lambda\in \nnopart(\cR)\},
\end{align*}
where
$$\delta_\mu(g)=\left\{\begin{array}{ll} 1 & \text{if $g\in \UT^\mu_\circ$},\\ 0 & \text{otherwise,}\end{array}\right.$$
is the superclass indicator function, and $\chi^\lambda$ is the supercharacter as in (\ref{PatternSupercharacterDefinition}).  These two bases are orthogonal with respect to the usual inner product
$$\langle \chi,\psi\rangle=\frac{1}{|\UT_\cR|} \sum_{u\in \UT_{\cR}} \chi(u)\overline{\psi(u)},$$
with
\begin{equation*}
\langle \delta_\mu,\delta_\nu\rangle =\delta_{\mu\nu} |\UT^\mu_\circ|\quad\text{and}\quad
\langle \chi^\lambda,\chi^\nu\rangle =\delta_{\lambda\nu} \chi^\lambda(1).
\end{equation*}

As a normal lattice supercharacter theory, we obtain a third canonical basis
$$\f(\UT_\cR;\Cl(\cR))=\CC\spanning\{\chi^{\UT_\lambda}\mid \lambda\in \nnopart(\cR)\},$$
where 
$$\chi^{\UT_\lambda}=\sum_{\UT_\nu\supseteq \UT_\lambda} \chi^{\nu}.$$
We might also be tempted by the dual construction
$$\delta_{\UT^\mu}=\sum_{\UT^\nu\subseteq \UT^\mu}\delta_\nu,$$
giving the normal subgroup indicator functions
\begin{equation*}
\delta_{\UT^\mu}(g)=\left\{\begin{array}{ll} 1 & \text{if $g\in \UT^\mu$},\\ 0 & \text{otherwise.}\end{array}\right.
\end{equation*}  
However, as the character of the permutation module $\Ind_{\UT_\lambda}^{\UT_\cR}(\One)$,
\begin{align*}
\chi^{\UT_\lambda}(g)&=\left\{\begin{array}{ll} \frac{|\UT_\cR|}{|\UT_\lambda|} & \text{if $g\in \UT_\lambda$},\\ 0 & \text{otherwise,}\end{array}\right.\\
&=\frac{|\UT_\cR|}{|\UT_\lambda|} \delta_{\UT_\lambda},
\end{align*}
so up to scaling these are in fact the same basis.  Neither version is orthogonal, and we have
$$\langle \chi^{\UT_\lambda},\chi^{\UT_\nu}\rangle=\frac{|\UT_\cR|}{|\UT_\lambda\UT_\nu|}\quad\text{and}\quad\langle \delta_{\UT^\mu},\delta_{\UT^\nu}\rangle=\frac{|\UT^\mu\cap \UT^\nu|}{|\UT_\cR|}.
$$
However, they give a upper/lower triangular decomposition of the supercharacter table of $\UT_\cR$ (similar to the situation in \cite{BT15} for the full upper-triangulars).  By keeping track of the corresponding diagonal entries, we obtain a formula for the determinant of the supercharacter table.

\begin{proposition}
Let $\SC(\cR)$ be the supercharacter table of $\UT_\cR$.  Then 
$$\det(\SC(\cR))=\prod_{\lambda\in \nnopart(\cR)} \frac{|\UT_\cR|}{|\UT_\lambda|}.$$
\end{proposition}

For these bases we can compute the structure constants as follows.  First the superclass indicators give

\begin{lemma}[\cite{An15}]
Let $A,B$ be sets with $A\cap B=\emptyset$. 
\begin{enumerate}
\item[(a)] For $\cR\in \poset(A)$, $\cQ\in \poset(B)$, $\mu\in \nnopart(\cR)$ and $\nu\in\nnopart(\cQ)$,
$$\Inf_{A,B}(\delta_\mu\otimes\delta_\nu)=\sum_{\lambda\in \nnopart(\cP.\cQ)\atop \lambda|_A=\mu,\lambda|_B=\nu}\delta_\lambda.$$
\item[(b)] For $\cP\in \poset(A\cup B)$, $\lambda\in \nnopart(\cP)$,
$$\Res_{A,B}(\delta_\lambda)=\left\{\begin{array}{@{}l@{\ }l} \delta_{\lambda_A}\otimes \delta_{\lambda_B} & \text{if $\lambda_A\cup \lambda_B=\lambda$,}\\ 0 &\text{otherwise.}\end{array}\right.$$
\end{enumerate}
\end{lemma}

The normal subgroup basis has the following constants.

\begin{lemma}\label{StructureSubgroupPoset}
Let $A,B$ be sets with $A\cap B=\emptyset$. 
\begin{enumerate}
\item[(a)] For $\cO,\cR\in \poset(A)$, $\cP,\cQ\in \poset(B)$ with $\UT_\cO\triangleleft \UT_\cR$ and $\UT_\cP\triangleleft \UT_\cQ$,
$$\Inf_{A,B}(\delta_{\UT_\cO}\otimes\delta_{\UT_\cP})=\delta_{\UT_{\cO.\cP}}.$$
\item[(b)] For $\cO,\cR\in \poset(A\cup B)$ with $\UT_\cO\triangleleft\UT_\cR$,
$$\Res_{A,B}(\delta_{\UT_\cO})=\delta_{\UT_{\cO|_A}} \otimes \delta_{\UT_{\cO|_B}}.$$
\end{enumerate}
\end{lemma}

\begin{proof}
(a) For $u\in \UT_{\cR.\cQ}$,
$$\Inf_{A,B}(\delta_{\UT_\cO}\otimes\delta_{\UT_\cP})(u)=\Big(\delta_{\UT_\cO}\otimes\delta_{\UT_\cP}\Big)(\pi_{A,B}(u)).$$
Since $\pi_{A,B}^{-1}(\UT_\cO\otimes \UT_\cP)=\UT_{\cO.\cP}$, the result follows.

(b) For $(u,v)\in \UT_{\cR|_A}\times \UT_{\cR|_B}$,
$$\Res_{A,B}(\delta_{\UT_\cO})(u,v)=\delta_{\UT_\cO}(\iota_{A,B}(u,v)).$$
Since $\iota_{A,B}^{-1}(\UT_{\cO})=\UT_{\cO|_A}\times \UT_{\cO|_B}$, the result follows.
\end{proof}

If we want the analogous result in terms of the combinatorics, we need to be a bit more careful. If $\UT_\cO=\UT^\lambda$ and $\UT_\cP=\UT^\nu$, then $\UT_{\cO.\cP}\neq \UT^{\lambda\cup \nu}$ in general.  However,  if $\UT_\cO=\UT_\lambda$ and $\UT_\cP=\UT_\nu$, then $\UT_{\cO.\cP}=\UT_{\lambda\cup \nu}$.

\begin{lemma}\label{StructureSubgroupPartition}
Let $A,B$ be sets with $A\cap B=\emptyset$. 
\begin{enumerate}
\item[(a)] For $\cR\in \poset(A)$, $\cQ\in \poset(B)$, $\lambda\in \nnopart(\cR)$ and $\nu\in\nnopart(\cQ)$,
$$\Inf_{A,B}(\chi^{\UT_\lambda}\otimes\chi^{\UT_\nu})=\chi^{\UT_{\lambda\cup\nu}}.$$
\item[(b)] For $\cR\in \poset(A\cup B)$, $\mu\in \nnopart(\cR)$,
$$\Res_{A,B}(\delta_{\UT^\mu})=\delta_{\UT^\mu\cap\UT_{\cR|_A}} \otimes \delta_{\UT^\mu\cap\UT_{\cR|_B}}.$$
\end{enumerate}
\end{lemma}

\begin{proof}
(a) For $u\in \UT_{\lambda\cup \nu}$,
$$\Inf_{A,B}(\chi^{\UT_\lambda}\otimes\chi^{\UT_\nu})(u)=\frac{|\UT_\cR|}{|\UT_\lambda|}\frac{|\UT_\cQ|}{|\UT_\nu|}\frac{q^{|A||B|}}{q^{|A||B|}}=\frac{\UT_{\cR.\cQ}}{|\UT_{\lambda\cup\nu}|}=\chi^{\UT_{\lambda\cup\nu}}(u).$$
(b) Here there is a less direct combinatorial description, but the result follows from Lemma \ref{StructureSubgroupPoset} (b).
\end{proof}

Somewhat surprisingly, we can also compute the structure constants for the supercharacter basis.  

\begin{theorem}\label{StructureConstantsSupercharacters}
Let $A,B$ be sets with $A\cap B=\emptyset$. 
\begin{enumerate}
\item[(a)] For $\cR\in \poset(A)$, $\cQ\in \poset(B)$, $\lambda\in \nnopart(\cR)$ and $\nu\in\nnopart(\cQ)$,
$$\Inf_{A,B}(\chi^{\lambda}\otimes\chi^{\nu})=\chi^{\lambda\cup\nu}.$$
\item[(b)] For $\cR\in \poset(A\cup B)$, $\lambda\in \nnopart(\cR)$,

$$\Res_{A,B}(\chi^{\lambda})=\frac{|\UT_\lambda\cap(\UT_{\cR|_A}\times\UT_{\cR|_B})|}{|\UT_{\lambda_A}\times \UT_{\lambda_B}|}\frac{\chi^{\lambda}(1)}{\chi^{\lambda_{A}\cup\lambda_{B}}(1)}\sum_{ \nu\in \anti(\Int^\lambda_{\lambda_A})\atop \eta\in \anti(\Int^\lambda_{\lambda_B})} \chi^{\lambda_A\cup\nu}\otimes \chi^{\lambda_B\cup\eta}.$$
\end{enumerate}
\end{theorem}
\begin{proof}
(a) follows from a similar argument in \cite{An15} and (b) follows directly from Corollary \ref{PatternRestriction} with Lemma \ref{RestrictionCombinatorics}.
\end{proof}

\subsection{Freeness}

The first goal of this section is to show that $\pattern_\NPL$ is a free  monoid.  Recall that a \textbf{\emph{set composition}} $(B_1,\ldots, B_\ell)\vDash B$ of a set $B$ is a sequence of nonempty sets $(B_1,\ldots, B_\ell)$ such that $\{B_1,\ldots,B_\ell\}$ is a set partition of $B$.

For a set $B$, $\cP,\cQ\in \poset(B)$ with $\UT_\cQ\triangleleft \UT_\cP$, we say the pair $(\cP,\cQ)$  \textbf{\emph{factors}} if  there exists a nonempty, proper subset $A\subseteq B$ such that 
$$\cP|_{A}.\cP|_{B-A}=\cP\quad \text{and}\quad \cQ|_A.\cQ|_{B-A}=\cQ.
$$
If no such subset exists, we say the pair $(\cP,\cQ)$  is \textbf{\emph{atomic}}.   

In general,  for a pair $(\cP,\cQ)$ there exists a set composition $(B_1,\ldots, B_\ell)\vDash B$ such that for each $1\leq j\leq \ell$, $(\cP|_{B_j},\cQ|_{B_j})$ is atomic and 
\begin{equation}\label{PoPosetDecomposition}
\cP=\cP|_{B_1}.\ldots .\cP|_{B_\ell}\quad \text{and}\qquad \cQ=\cQ|_{B_1}.\ldots .\cQ|_{B_\ell}.
\end{equation}

Define a vector species $\poposet:\{\text{sets}\}\rightarrow \{\CC\text{-modules}\}$ by
$$\poposet[A]=\left\{\begin{array}{@{}ll} \dd\bigoplus_{\cR\in \poset(A)} \CC\spanning\{(\cR,\cQ)\text{ atomic}\mid \UT_\cQ\triangleleft \UT_\cR\} & \text{if $A\neq \emptyset$,}\\ 0 & \text{otherwise,}\end{array}\right.$$
where we consider the pairs $(\cR,\cQ)$ formally linearly independent.   We upgrade this vector species into a free connected Hopf monoid $\fpoposet: \{\text{sets}\}\rightarrow \{\CC\text{-modules}\}$ by
$$\fpoposet[A]=\left\{\begin{array}{@{}ll} \dd\bigoplus_{(A_1,\ldots, A_\ell)\vDash A\atop \ell\geq 1} \poposet[A_1]\otimes \cdots \otimes \poposet[A_\ell] & \text{if $A\neq \emptyset$,}\\ \CC & \text{otherwise,}\end{array}\right.$$
where each $(A_1,\ldots,A_\ell)\vDash A$ is a set composition of $A$. The product in $\fpoposet$ is concatenation of tensors, and for nonintersecting sets $A$ and $B$ and $(\cP,\cQ)\in\poposet[A\cup B]$,
$$   \Delta_{A,B}((\cP,\cQ))=\left\{\begin{array}{ll} 0 & \text{if $A,B\neq \emptyset$,}\\ 
(\cP,\cQ)\otimes 1 & \text{if $B=\emptyset$,}\\
1\otimes (\cP,\cQ) & \text{if $A=\emptyset$}.
\end{array}\right.$$
That is, $(\cP,\cQ)$ are primitive.

Extend the function
$$\begin{array}{ccc} \poposet[A] & \longrightarrow & \f\Big(\UT_\cP;\Cl(\cP)\Big)\\
(\cP,\cQ) & \mapsto & \delta_{\UT_\cQ}\end{array}$$
multiplicatively to a Hopf isomorphism 
$$\begin{array}{cccl} \fpoposet  & \longrightarrow &  \pattern_{ \NPL } & \\
(\cP_1,\cQ_1) \otimes \cdots \otimes (\cP_\ell, \cQ_\ell) & \mapsto  & \delta_{\UT_{\cQ_1.\ldots .\cQ_\ell}} & \Big(\in \f(\UT_{\cP_1.\ldots.\cP_\ell};\Cl(\cP_1.\ldots.\cP_\ell))\Big)\\
(\cP|_{B_1},\cQ|_{B_1})\otimes \cdots \otimes (\cP|_{B_\ell},\cQ|_{B_\ell})  & \mapsfrom & \delta_{\UT_\cQ} & \Big(\in \f(\UT_\cP;\Cl(\cP))\Big)
\end{array}$$
where $(B_1,\ldots, B_\ell)$ is the set composition coming from the unique decomposition (\ref{PoPosetDecomposition}).
We have proven the following proposition.

\begin{proposition}
The Hopf monoid $\pattern_\NPL$ is a free monoid.
\end{proposition}

\subsection{Antipodes}

For a set $B$, $\cP,\cQ\in \poset(B)$, a set composition $(I_1,\ldots, I_\ell)\vDash B$ is a $\cQ$\textbf{\emph{-factorization}} of $\cP$ if 
\begin{enumerate}
\item[(a)] $\cP|_{I_j}$ is a convex subposet of $\cQ$,
\item[(b)]  all minimal elements of $\cP|_{I_j}$ are greater than the maximal elements of $\cP|_{I_{j-1}}$.
\end{enumerate}
The number $\ell_\cQ(I_1,\ldots, I_\ell)=\ell$ is the  \textbf{\emph{length}} of the factorization.
Let
$$\Fac_\cQ(\cP)=\{\text{$\cQ$-factorizations of $\cP$}\}.$$
For example, if 
$$\cQ=\begin{tikzpicture} [scale=.3,baseline=1]
\foreach \x/\y/\z/\w in {0/0/\heartsuit/1,-.5/1/\diamondsuit/2,.5/1/\clubsuit/3,0/2/\spadesuit/4}
	\node (\w) at (\x,\y)[inner sep=0pt] {$\scriptscriptstyle\z$};
\foreach \x/\y in {1/2,1/3,2/4,3/4}
	\draw (\x) -- (\y);
\end{tikzpicture},\cP= \begin{tikzpicture} [scale=.3,baseline=1.5]
\foreach \x/\y/\z/\w in {0/0/\heartsuit/1,0/1/\diamondsuit/2,0/2/\clubsuit/3,0/3/\spadesuit/4}
	\node (\w) at (\x,\y)[inner sep=0pt] {$\scriptscriptstyle\z$};
\foreach \x/\y in {1/2,2/3,3/4}
	\draw (\x) -- (\y);
\end{tikzpicture}, \qquad\text{then}\qquad
\Fac_\cQ(\cP)=\bigg\{\Big( \heartsuit,\diamondsuit,\clubsuit,\spadesuit\Big),\Big( \begin{tikzpicture} [scale=.4,baseline=1.5]
\foreach \x/\y/\z/\w in {0/0/\heartsuit/1,0/1/\diamondsuit/2}
	\node (\w) at (\x,\y)[inner sep=0pt] {$\scriptstyle\z$};
\foreach \x/\y in {1/2}
	\draw (\x) -- (\y);
\end{tikzpicture},\clubsuit,\spadesuit\Big),\Big( \heartsuit,\diamondsuit, \begin{tikzpicture} [scale=.4,baseline=1.5]
\foreach \x/\y/\z/\w in {0/0/\clubsuit/1,0/1/\spadesuit/2}
	\node (\w) at (\x,\y)[inner sep=0pt] {$\scriptstyle\z$};
\foreach \x/\y in {1/2}
	\draw (\x) -- (\y);
\end{tikzpicture}\Big),\Big( \begin{tikzpicture} [scale=.4,baseline=1.5]
\foreach \x/\y/\z/\w in {0/0/\heartsuit/1,0/1/\diamondsuit/2}
	\node (\w) at (\x,\y)[inner sep=0pt] {$\scriptstyle\z$};
\foreach \x/\y in {1/2}
	\draw (\x) -- (\y);
\end{tikzpicture},\begin{tikzpicture} [scale=.4,baseline=1.5]
\foreach \x/\y/\z/\w in {0/0/\clubsuit/1,0/1/\spadesuit/2}
	\node (\w) at (\x,\y)[inner sep=0pt] {$\scriptstyle\z$};
\foreach \x/\y in {1/2}
	\draw (\x) -- (\y);
\end{tikzpicture}\Big)\bigg\}.$$

By Takeuchi's formula and Lemma \ref{StructureSubgroupPoset} we have that for $\cQ\in \poset(A)$,
$$S(\delta_{\UT_\cQ})=\sum_{\cP\in \poset(A)} \Big(\sum_{\vec{I}\in \Fac_\cQ(\cP)} (-1)^{\ell_\cQ(\vec{I})-1} \Big) \delta_{\UT_\cP}.$$
Note that this formula does not depend on the ambient subgroup $\UT_\cR$ in which $\UT_\cQ$ is normal.  In fact, we obtain some cancellation.  The following result is similar to the power-sum result in \cite{BBT13}.

\begin{proposition} For $\cQ\in \poset(A)$,
$$S(\delta_{\UT_\cQ})=\sum_{\cP\in \poset(A)\atop \Fac_\cQ(\cP)=\{\vec{I}\}}  (-1)^{\ell_\cQ(\vec{I})-1}  \delta_{\UT_\cP}.$$
\end{proposition}

\begin{proof}  Suppose $\Fac_\cQ(\cP)\neq \emptyset$.  Let $\vec{I}\in \Fac_\cQ(\cP)$.  If there exists $I_j\in \vec{I}$ such that $\Fac_{\cQ|_{I_j}}(\cP|_{I_j})\neq \emptyset$, we may replace $I_j$ by any element of $\Fac_{\cQ|_{I_j}}(\cP|_{I_j})$.  If we continue in this way until no part splits further we obtain an element $\vec{L}=(L_1,\ldots, L_\ell)\in \Fac_{\cQ}(\cP)$, so
$$\cP=L_1\prec_\cP L_2 \prec_\cP \cdots\prec_\cP L_\ell,
$$
where for each $1\leq j\leq \ell$, every subset $A\subseteq L_j$ satisfies if  $(a,b)\in A\times (L_j-A)$ with $a\prec b$, then there exists $(a',b')\in A\times (L_j-A)$ such that $a'\nprec b'$.  It follows that every element in $\vec{I}\in \Fac_{\cQ}(\cP)$ there is a set partition $\{A_1,\ldots,A_k\}$ of $\{1,2,\ldots,\ell\}$ with $A_1<A_2<\cdots<A_k$ such that 
$$\vec{I}=(L_{A_1},\dots,L_{A_k}),\quad \text{where}\quad L_{A_i}=\bigcup_{j\in A_i} L_j.$$

On the other hand, there is a coarsest set partition $\{C_1,\ldots,C_k\}$ of $\{1,2,\ldots,\ell\}$ such that 
$$\vec{M}=(L_{C_1},\dots,L_{C_k})\in \Fac_\cQ(\cP).$$
If we assign a binary string $(a_1,\ldots,a_{\ell-1})\in \{0,1\}$ for each set partition $\{A_1,\ldots,A_m\}$, by the rule 
$$a_j=\left\{\begin{array}{ll} 1 & \text{if $j$ and $j+1$ are in different parts}\\ 0 & \text{if $j$ and $j+1$ are in the same part.}\end{array}\right.$$
Then $\Fac_\cQ(\cP)$ is isomorphic to an interval in the poset of binary strings of length $\ell-1$, where the top element is $(1,1,\ldots,1)$ and the bottom element corresponds to  $(L_{C_1},\dots,L_{C_k})$. But then
$$\sum_{\vec{I}\in \Fac_\cQ(\cP)} (-1)^{\ell_\cQ(\vec{I})-1}=(-1)^{\ell_\cQ(\vec{M})-1}\mu(\vec{M},\vec{L}),$$
where $\mu$ is the M\"obius function.  This is nonzero exactly when $\vec{M}=\vec{L}$.
\end{proof} 

We also obtain a formula for the supercharacter basis.  For $\cR\in \poset(B)$, $\lambda\in \nnopart(\cR)$ and a set composition $\vec{A}=(A_1,\ldots, A_\ell)\vDash B$, let
\begin{align*}
\cR|_{\vec{A}} &= \cR|_{A_1}\cup \cR|_{A_2}\cup \cdots \cup  \cR|_{A_\ell}\\
\UT_{\cR|_{\vec{A}}} & = \UT_{\cR|_{A_1}}\times \cdots \times \UT_{\cR|_{A_\ell}}\\
\lambda_{\vec{A}} &=\lambda_{A_1}\cup\cdots\cup \lambda_{A_\ell}\in \nnopart(\cR|_{\vec{A}})\\
\UT_{\lambda_{\vec{A}}}&=\UT_{\lambda_{A_1}}\times \cdots \times \UT_{\lambda_{A_\ell}}\\
\cR_\lambda&\subseteq \cR \quad \text{where $\UT_{\cR_\lambda}=\UT_\lambda$}.
\end{align*}

Note that by Corollary \ref{PatternRestriction} and Lemma \ref{RestrictionCombinatorics} we have that if $\nu\notin \Int^{\lambda}_{\lambda_{\vec{A}}}\subseteq \Int(\cR|_{\vec{A}})\subseteq \Int(\cR)$ for some $\vec{A}\vDash B$, then $\langle S(\chi^\lambda),\chi^\nu\rangle=0$.  Let
\begin{align*}
\Res^\lambda &=\bigcup_{\vec{A}\vDash B} \{\nu\cup \lambda_{\vec{A}}\mid \nu\in\anti\Big(\Int^{\lambda}_{\lambda_{\vec{A}}}\Big)\}\\
&=\{\nu\in \nnopart(\cR)\mid (\nu-\lambda)\cap\bigcap_{[i,l]\in \lambda -\nu} \cR_{[i,l]}=\emptyset\}.
\end{align*}
For $\lambda\in \nnopart(\cR)$, a $\cR$-factorization $(I_1,\ldots, I_\ell)\in \Fac_\cR(\cP)$ has a \textbf{\emph{$\lambda$-neutral cut}} if there exists $1\leq j\leq \ell-1$ such that 
$$\{[a,b]\mid a\in \max\{I_j\},b\in \min\{I_{j+1}\}\}\subseteq \Int^\circ(\cR_\lambda).$$
We say $\Fac_\cR(\cP)$ is \textbf{\emph{$\lambda$-atomic}} if it is nonempty and the longest element has no $\lambda$-neutral cuts.

\begin{theorem}
For $\cR\in \poset(B)$, and $\lambda\in \nnopart(\cR)$,
$$S(\chi^\lambda)=\sum_{\cP\in \poset(B)\atop \Fac_\cR(\cP)\text{ $\lambda$-atomic}} \sum_{\nu\in \nnopart(\cP)\cap \Res^\lambda} \frac{\chi^\lambda(1)}{(q-1)^{|\lambda\cap \nu|}}\Big(\sum_{{\vec{I}\in \Fac_\cR(\cP)\atop \nu_{\vec{I}} =\nu}\atop \lambda_{\vec{I}}=\lambda\cap \nu}(-1)^{\ell(\vec{I})-1}\frac{|\UT_\lambda\cap\UT_{\cR|_{\vec{I}}}|}{|\UT_{\cR|_{\vec{I}}}|}\Big) \chi^\nu.  $$
Furthermore, as polynomials in $q$, the coefficients are nonzero (though they may have integral roots).
\end{theorem}

\begin{remark}
An example that shows the coefficients can be zero is as follows.  The coefficient of $\emptyset\in \nnopart(1<2<3)$ in $S(\chi^{\{[1,3]\}})$ is 
$(q-1)(q-2)$, which is generically nonzero, but zero if $q=2$.  However, if $q$ is sufficiently large the coefficients are always nonzero.
\end{remark} 

\begin{proof} By Takeuchi's formula and Lemma \ref{RestrictionCombinatorics},
\begin{align*}
S(\chi^\lambda)&=\sum_{\vec{A}=(A_1,\ldots, A_\ell)\vDash B} (-1)^{\ell-1} m_{\vec{A}}\circ \Delta_{\vec{A}}(\chi^\lambda)\\
&=\sum_{\vec{A}=(A_1,\ldots, A_\ell)\vDash B} (-1)^{\ell-1} \frac{|\UT_\lambda\cap \UT_{\cR|_{\vec{A}}}|}{|\UT_{\lambda_{\vec{A}}}|}\frac{\chi^\lambda(1)}{\chi^{\lambda_{\vec{A}}}(1)}\sum_{{\nu=(\nu_1,\ldots,\nu_\ell)\in \nnopart(\cR|_{\vec{A}})\atop \nu-\lambda_{\vec{A}}\in \Int_{\lambda_{\vec{A}}}^\lambda}\atop \text{an anti-chain} } \chi^{\nu_1.\nu_2.\cdots.\nu_\ell}. 
\end{align*}
Fix $\nu\in \Res^\lambda$ and $\cP\in \poset(B)$ such that $\Fac_\cR(\cP)\neq \emptyset$ and $\nu\in \nnopart(\cP)$.  Then the coefficient of $\chi^\nu$ in $S(\chi^\lambda)$ is
\begin{align*}
\sum_{{\vec{I}\in \Fac_\cR(\cP)\atop \nu_{\vec{I}} =\nu}\atop \lambda_{\vec{I}}=\lambda\cap \nu} (-1)^{\ell(\vec{I})-1} \frac{|\UT_\lambda\cap \UT_{\cR|_{\vec{I}}}|}{|\UT_{\lambda_{\vec{I}}}|}\frac{\chi^\lambda(1)}{\chi^{\lambda_{\vec{I}}}(1)}&=\chi^\lambda(1)\sum_{{\vec{I}\in \Fac_\cR(\cP)\atop \nu_{\vec{I}} =\nu}\atop \lambda_{\vec{I}}=\lambda\cap \nu} (-1)^{\ell(\vec{I})-1} \frac{|\UT_\lambda\cap \UT_{\cR|_{\vec{I}}}|}{\chi^{\lambda\cap \nu}(1)|\UT_{\lambda\cap\nu}|}\\
&=\frac{\chi^\lambda(1)}{(q-1)^{|\lambda\cap\nu|}}\sum_{{\vec{I}\in \Fac_\cR(\cP)\atop \nu_{\vec{I}} =\nu}\atop \lambda_{\vec{I}}=\lambda\cap \nu} (-1)^{\ell(\vec{I})-1} \frac{|\UT_\lambda\cap \UT_{\cR|_{\vec{I}}}|}{|\UT_{\cR|_{\vec{I}}}|}.
\end{align*}
Let $\vec{L}\in \Fac_\cR(\cP)$ be the longest element.  If $\vec{L}$ has a $\lambda$-neutral cut, then $\cP$ can be split into $\cP|_A\prec_\cP \cP|_{A'}$,
such that each pair $[a,b]\in \Int^\circ(\cR_\lambda)$ for $a$ a maximal element of $A$ and $b$ a minimal element of $A'$.  In particular, if $[i,j]\in \lambda$ then either $\{i,j\}\subseteq A$ or $\{i,j\}\subseteq A'$.  Furthermore, if $[i,j]\in \nu$, then $\{i,j\}\subseteq A$ or $\{i,j\}\subseteq A'$.  Thus, the presence of the $\lambda$-neutral cut in $\vec{I}$ does not affect whether $\vec{I}$ satisfies  $\nu_{\vec{I}} =\nu$ and $\lambda_{\vec{I}}=\lambda\cap \nu$.  Thus, we have a bijection
$$\begin{array}{c}\{\vec{I}\in \Fac_\cR(\cP)\mid \vec{I}|_A\cup \vec{I}|_{A'}=\vec{I},\lambda_{\vec{I}}=\lambda\cap \nu\}\\
\updownarrow\\ \{\vec{I}\in \Fac_\cR(\cP)\mid \vec{I}|_A\cup \vec{I}|_{A'}\neq   \vec{I},\nu_{\vec{I}} =\nu,\lambda_{\vec{I}}=\lambda\cap \nu\}\end{array}$$
obtained by either adding or removing this $\lambda$-neutral cut.  Furthermore, since all pairs between $A$ and $A'$ are in $\Int(\cR_\lambda)$, the existence of the fixed $\lambda$-neutral cut does not affect the value of 
$$\frac{|\UT_\lambda\cap \UT_{\cR|_{\vec{I}}}|}{|\UT_{\cR|_{\vec{I}}}|}.$$
Thus, the bijection is sign reversing and the coefficient is zero if $\Fac_\cR(\cP)$ is not $\lambda$-atomic.

Suppose $\Fac_\cR(\cP)$ is $\lambda$-atomic.  If there exists $\vec{I}\in \Fac_\cR(\cP)$ such that $\nu_{\vec{I}} =\nu$ and $\lambda_{\vec{I}}=\lambda\cap \nu$, then there exists a unique longest such element $\vec{L}$. Then 
$$\sum_{{\vec{I}\in \Fac_\cR(\cP)\atop \nu_{\vec{I}} =\nu}\atop \lambda_{\vec{I}}=\lambda\cap \nu}\hspace{-.25cm}(-1)^{\ell(\vec{I})-1}\frac{|\UT_\lambda\cap\UT_{\cR|_{\vec{I}}}|}{|\UT_{\cR|_{\vec{I}}}|}=(-1)^{\ell(\vec{L})-1}\frac{|\UT_\lambda\cap\UT_{\cR|_{\vec{L}}}|}{|\UT_{\cR|_{\vec{L}}}|} + \hspace{-.25cm}\sum_{{\vec{L}\neq\vec{I}\in \Fac_\cR(\cP)\atop \nu_{\vec{I}} =\nu}\atop \lambda_{\vec{I}}=\lambda\cap \nu}\hspace{-.25cm}(-1)^{\ell(\vec{I})-1}\frac{|\UT_\lambda\cap\UT_{\cR|_{\vec{I}}}|}{|\UT_{\cR|_{\vec{I}}}|}.$$
Since $\vec{L}$ has no $\lambda$-neutral cuts, any set composition in $\vec{I}\in \Fac_\cR(\cP)$ such that $\vec{I}\neq \vec{L}$ will add elements in $\UT_{\cR|_{\vec{I}}}$ that are not in  $\UT_\lambda\cap\UT_{\cR|_{\vec{I}}}$.  Thus,
$$\frac{|\UT_\lambda\cap\UT_{\cR|_{\vec{L}}}|}{|\UT_{\cR|_{\vec{L}}}|} >\frac{|\UT_\lambda\cap\UT_{\cR|_{\vec{I}}}|}{|\UT_{\cR|_{\vec{I}}}|}.$$
Since all the quotients are powers of $q$, there exists $k\in \ZZ_{\geq 0}$ and $c_l\in \ZZ$ such that 
$$\sum_{{\vec{I}\in \Fac_\cR(\cP)\atop \nu_{\vec{I}} =\nu}\atop \lambda_{\vec{I}}=\lambda\cap \nu}\hspace{-.25cm}(-1)^{\ell(\vec{I})-1}\frac{|\UT_\lambda\cap\UT_{\cR|_{\vec{I}}}|}{|\UT_{\cR|_{\vec{I}}}|}=\pm\Big(\frac{1}{q^k}+\sum_{l>k} \frac{c_l}{q^l}\Big),\quad\text{where}\quad \sum_{l>k} |c_l|=|\Fac_\cR(\cP)|-1.$$
Generically this polynomial is nonzero, though for specific values of $q$ it could have a root. 
\end{proof}

We can apply the theorem to the specific case of trivial characters to get a pleasing result.   For $\cP\in \poset(B)$ let $\emptyset_\cP=\emptyset\in \nnopart(\cP)$.

\begin{corollary}  For $\cR\in \poset(B)$,    
$$S(\chi^{\emptyset_\cR})=\sum_{\cP\in \poset(B)\atop \Fac_\cR(\cP)=\{\vec{L}\}} (-1)^{\ell(\vec{L})-1}\chi^{\emptyset_{\cP}}.$$
\end{corollary}
\begin{proof}
Since $\UT_{\emptyset_\cR}=\UT_\cR$, we have that $\UT_{\emptyset_\cR}\cap \UT_{\cR|_{\vec{A}}}=\UT_{\cR|_{\vec{A}}}$ for all $\vec{A}\vDash B$.  Thus, if $\Fac_\cR(\cP)$ is $\emptyset_\cR$-atomic, then for the longest element $\vec{L}=(L_1,\ldots,L_\ell)\in \Fac_\cR(\cP)$ we must have that for each $1\leq j\leq \ell-1$ at least one maximal elements of $L_j$ is greater in $\cR$ than a minimal element in $L_{j+1}$.  Thus, $\vec{L}$ is also the minimal element in $\Fac_\cR(\cP)$.   The other elements of the formula reduce in a straightforward way.
\end{proof}

\subsection{Primitives}

In \cite{AL15}, the authors indicate how to find the dimension of  the lie algebra of primitive elements of a (co)free Hopf monoid in each degree. For $\pattern_{\NPL}$ we go one step forward and give a full system of algebraic independent primitive elements.
Fix $a\in A$.  For $\cQ$ atomic, define
$$S^{(a)}_{\UT_\cQ}=\sum_{\Fac_\cQ(\cP)=\{\vec{I}\}\atop a\in I_1} (-1)^{\ell_\cQ(\vec{I})-1}\delta_{\UT_\cP}.$$ 

\begin{theorem}
Fix a functor $\select:\{\text{sets}\}\rightarrow \{\text{sets of size 1}\}$ such that $\select(A)\subseteq A$. 
Then 
$$\{S^{(\select(A))}_{\UT_\cQ}\mid \UT_\cQ\triangleleft \UT_\cR, \cR\in \poset(A), (\cQ,\cR)\text{ atomic}\}$$
is a full system of algebraically independent primitive elements in $\pattern_{\NPL}$.
\end{theorem}

\begin{proof}
By \cite[Theorem 7.1]{BBT13}, each $S^{(a)}_{\UT_\cQ}$ is projective, and by definition they are nonzero.  Next note that  
$$S^{(a)}_{\UT_\cQ}=\delta_{\UT_\cQ}+\sum_{\Fac_\cQ(\cP)=\{\vec{I}\},\cP\neq \cQ\atop a\in I_1} (-1)^{\ell_\cQ(\vec{I})-1}\delta_{\UT_\cP}.$$ 
If we put an ordering on $\poset(A)$ by $\cQ\preceq \cR$ if $|\UT_\cQ|\leq |\UT_\cR|$. Since $\Fac_\cQ(\cR)\neq \emptyset$ implies $\cR$ has at least as many relations as $\cQ$, it also implies that $|\UT_\cQ|\leq |\UT_\cR|$ with equality only when $\cQ=\cR$.  Thus, $S^{(a)}_{\UT_\cQ}$ are triangular in the $\delta_{\UT_\cR}$ basis, with leading term the algebraically independent elements $\delta_{\UT_\cQ}$.  Conversely, given $\cR\in \poset(A)$, there exists a factorization into atomics 
$$\cR=\cQ_1.\cQ_2.\cdots.\cQ_\ell\quad \text{with}\quad \cQ_j\in \poset(A_j)$$
and 
$$S^{(\select(A_1))}_{\UT_{\cQ_1}}S^{(\select(A_2))}_{\UT_{\cQ_2}}\cdots S^{(\select(A_\ell))}_{\UT_{\cQ_\ell}}=\delta_{\UT_\cR}+\sum_{\cS\in \poset(A)\atop \UT_\cS>\UT_\cR} c_\cR^\cS \delta_{\UT_\cS}.$$
Thus, the $S^{(a)}_{\UT_\cQ}$ generate a basis of $\pattern_{\NPL}$ and are a full set of algebraically independent projective elements.
\end{proof}

In \cite{NT}, Novelli--Thibon define a graded Hopf algebra 
$$\mathsf{CQSym}=\bigoplus_{n\geq 0}\mathsf{CQSym}_n $$
with $\dim(\mathsf{CQSym}_n)$ equal to the $n$th Catalan number $C_n$.  In the case of this paper they index a basis in each graded dimension by non-decreasing parking functions, which are easily seen to be in bijection with non-crossing partitions (which in turn are in bijection with non-nesting partitions).   They also prove that $\mathsf{CQSym}$ is free over a set of primitive generators indexed by connected parking functions (parking functions that cannot be written as a shifted concatenation of two other parking functions).   For degree $n$, these are enumerated by
$$\sum_{j=2}^nC_{j-2}C_{n-j}=C_{n-1}.$$

\begin{corollary}
The Hopf algebra $\mathsf{CQSym}$ is isomorphic to $\mathrm{NCSym}_\nn$.
\end{corollary}

\begin{proof}
We have that each algebra is free under the product of primitive (atomic) generators.  The number of primitive generators in degree $n$ for $\mathrm{NCSym}_\nn$ is the number of non-nesting partitions that that are not the concatenation of two other nontrivial  non-nesting partitions.  However, a non-nesting partition is atomic if and only if the corresponding non-crossing partition (the unique partition that has the same left and right endpoints) is atomic.  The number of atomic non-crossing partitions of $\{1,2,\ldots,n\}$ is  
$$\sum_{j=2}^nC_{j-2}C_{n-j}=C_{n-1}.$$
Thus, the two Hopf algebras must be isomorphic.
\end{proof}


\vspace{2cm}
\noindent (Farid Aliniaeifard) Department of Mathematics, University of Colorado at Boulder, Boulder, Colorado, 80309, USA\\
E-mail address: farid.aliniaeifard@colorado.edu\\
\\
\noindent (Nathaniel Thiem) Department of Mathematics, University of Colorado at Boulder, Boulder, Colorado, 80309, USA.\\
E-mail address: thiemn@colorado.edu

\end{document}